\documentclass[a4paper, draft, 12pt]{amsproc}
\usepackage{amssymb}
\usepackage{amsmath,amssymb,ulem,color}

\usepackage[hyphens]{url} 
\usepackage[dvips]{graphicx} 
\usepackage{cmdtrack} 
\usepackage{mathrsfs}                       

\theoremstyle{plain}
 \newtheorem{theorem}{Theorem}[section]
 \newtheorem{prop}{Proposition}[section]
 \newtheorem{lem}{Lemma}[section]
 
\theoremstyle{Definition}
 \newtheorem{exm}{Example}[section]
 
 \newtheorem{dfn}{Definition}[section]
\theoremstyle{remark}
 \newtheorem{rem}{Remark}[section]
 \newtheorem{conj}{Conjecture}[section]
 \numberwithin{equation}{section}

\renewcommand{\leq}{\leqslant}
\renewcommand{\geq}{\geqslant}
\renewcommand{\setminus}{\smallsetminus}
\title[Fatou Set, Julia Set and Escaping Set in Holomorphic Sub...]{Fatou Set, Julia Set and Escaping Set in Holomorphic Subsemigroup Dynamics}



\author[B. H. Subedi]{\bfseries  Bishnu Hari Subedi}

\address{ 
Central Department of Mathematics \\ 
Institute of Science and Technology   \\ 
Tribhuvan University   \\ 
Kirtipur, Kathmandu\\
Nepal}
\email{subedi.abs@gmail.com / subedi\_bh@cdmathtu.edu.np }

\author[A. Singh]{Ajaya Singh}
\address{Central Department of Mathematics, Institute of Science and Technology, Tribhuvan University, Kirtipur, Kathmandu, Nepal }
\email{singh.ajaya1@gmail.com / singh\_a@cdmathtu.edu.np} 
\thanks{This research work of first author is supported from PhD faculty fellowship of University Grants Commission, Nepal.} 

\begin{document}

{\begin{flushleft}\baselineskip9pt\scriptsize
MANUSCRIPT
\end{flushleft}}
\vspace{18mm} \setcounter{page}{1} \thispagestyle{empty}

\begin{abstract}
We investigate to what extent Fatou set, Julia set and  escaping set of transcendental semigroup is respectively equal to the Fatou set, Julia set and escaping set of its subsemigroup. We define partial fundamental set and fundamental set of transcendental semigroup and on the basis of this set, we prove that Fatou set and escaping set of transcendental semigroup $ S $ are non-empty. 

\end{abstract}

\maketitle
\section{Introduction}
We confine our study on Fatou set, Julia set and escaping set of holomorphic semigroup and its subsemigroup defined in complex plane $ \mathbb{C} $ or extended complex plane $ \mathbb{C}_{\infty} $. Semigroup $ S $ is a very classical algebraic structure with a binary composition that satisfies associative law. It naturally arose from the general mapping of a set into itself. So a set of holomorphic maps on $ \mathbb{C} $ or $ \mathbb{C}_{\infty} $ naturally forms a semigroup. Here, we take a set $ A $ of holomorphic maps and construct a semigroup $ S $ consists of all elements that can be expressed as a finite composition of elements in $ A $. We call such a semigroup $ S $ by \textit{holomorphic semigroup} generated by set $ A $. A non-empty subset $ T $ of holomorphic semigroup $ S $ is a \textit{subsemigroup} of $ S $ if $ f \circ g \in T $ for all $ f, \; g \in T $.  

For our simplicity, we denote the class of all rational maps on $ \mathbb{C_{\infty}} $ by $ \mathscr{R} $ and class of all transcendental entire maps on $ \mathbb{C} $ by $ \mathscr{E} $. 
Our particular interest is to study of the dynamics of the families of above two classes of holomorphic maps.  For a collection $\mathscr{F} = \{f_{\alpha}\}_{\alpha \in \Delta} $ of such maps, let 
$$
S =\langle f_{\alpha} \rangle
$$ 
be a \textit{holomorphic semigroup} generated by them. Here $ \mathscr{F} $ is either a collection $ \mathscr{R} $ of rational maps or a collection $ \mathscr{E} $ of transcendental entire maps (there are several  holomorphic semigroups generated by general meromorphic functions,  but in this paper, we are only  interested in holomorphic semigroups generated by either rational functions or transcendental entire functions). The index set $ \Delta $ to which $ \alpha $  belongs is allowed to be infinite in general unless otherwise stated. 
Here, each $f \in S$ is a holomorphic function and $S$ is closed under functional composition. Thus, $f \in S$ is constructed through the composition of finite number of functions $f_{\alpha_k},\;  (k=1, 2, 3,\ldots, m) $. That is, $f =f_{\alpha_1}\circ f_{\alpha_2}\circ f_{\alpha_3}\circ \cdots\circ f_{\alpha_m}$. In particular,  if $ f_{\alpha} \in \mathscr{R} $, we say $ S =\langle f_{\alpha} \rangle$ a \textit{rational semigroup} and if  $ f_{\alpha} \in \mathscr{E} $, we say $ S =\langle f_{\alpha} \rangle$ a \textit{transcendental semigroup}. 

A semigroup generated by finitely many holomorphic functions $f_{i}, (i = 1, 2, \ldots, \\ n) $  is called \textit{finitely generated  holomorphic semigroup}. We write $S= \langle f_{1},f_{2},\ldots,f_{n} \rangle$.
 If $S$ is generated by only one holomorphic function $f$, then $S$ is \textit{cyclic semigroup}. We write $S = \langle f\rangle$. In this case, each $g \in S$ can be written as $g = f^n$, where $f^n$ is the nth iterates of $f$ with itself. Note that in our study of  semigroup dynamics, we say $S = \langle f\rangle$  a \textit{trivial semigroup}. 
 
 Next, we define and discuss some special collection and sequences of holomorphic functions.
Note that all notions of convergence that we deal in this paper will be with respect to the Euclidean metric on the complex plane $ \mathbb{C} $ or spherical metric on the Riemann sphere $ \mathbb{C}_{\infty} $.

The family $\mathscr{F}$  of complex analytic maps forms a \textit{normal family} in a domain $ D $ if given any composition sequence $ (f_{\alpha}) $ generated by the member of  $ \mathscr{F} $,  there is a subsequence $( f_{\alpha_{k}}) $ which is uniformly convergent or divergent on all compact subsets of $D$. If there is a neighborhood $ U $ of the point $ z\in\mathbb{C} $ such that $\mathscr{F} $ is normal family in $U$, then we say $ \mathscr{F} $ is normal at $ z $. If  $\mathscr{F}$ is a family of members from the semigroup $ S $, then we simply say that $ S $ is normal in the neighborhood of $ z $ or $ S $ is normal at $ z $.

Let  $ f $ be a holomorphic map. We say that  $ f $ \textit{iteratively divergent} at $ z \in \mathbb{C} $ if $  f^n(z)\rightarrow \alpha \; \textrm{as} \; n \rightarrow \infty$, where $ \alpha $  is an essential singularity of $ f $. A sequence $ (f_{k})_{k \in \mathbb{N}} $ of holomorphic maps is said to be \textit{iteratively divergent} at $ z $ if $ f_{k}^{n}(z) \to\alpha_{k} \;\ \text{as}\;\ n\to \infty$ for all $ k \in \mathbb{N} $, where $ \alpha_{k} $  is an essential singularity of $ f_{k} $ for each $ k $.  Semigroup $ S $ is \textit{iteratively divergent} at $ z $ if $f^n(z)\rightarrow \alpha_{f} \; \textrm{as} \; n \rightarrow \infty$, where $ \alpha_{f} $  is an essential singularity of each $ f \in S $. Otherwise, a function $ f  $, sequence $ (f_{k})_{k \in \mathbb{N}} $ and semigroup $ S $  are said to be \textit{iteratively bounded} at $ z $. 
The following result will be clear from the definition of holomorphic semigroup. It shows that every element of holomorphic semigroup can be written as finite composition of the sequence of $f_{\alpha} $
\begin{prop}\label{ts1}
 Let $S   = \langle f_{\alpha} \rangle$ be an arbitrary  holomorphic semigroup. Then for every $ f \in S $,  $f^{m} $(for all $ m \in \mathbb{N}$) can be written as $f^{m} =f_{\alpha_1}\circ f_{\alpha_2}\circ f_{\alpha_3}\circ \cdots\circ f_{\alpha_p}$ for some $ p\in \mathbb{N} $.
 \end{prop}

In classical complex dynamics, each of Fatou set, Julia set and escaping set are defined in two different but equivalent ways.  In first definition, Fatou set is defined as the set of normality of the iterates of given function, Julia set is defined as the complement of the Fatou set and escaping set is defined as the set of points that goes to essential singularity under the iterates of given function. The second definition of  Fatou set is given as a largest completely invariant open set and Julia set is given as a smallest completely invariant close set  whereas escaping set is a completely invariant non-empty neither open nor close set in $\mathbb{C} $. 
Each of these definitions can be naturally extended to the settings of holomorphic semigroup $ S $ but extension definitions are not equivalent.  
Based on above first definition (that is,  on the Fatou-Julia-Eremenko theory of a complex analytic function), the Fatou set, Julia set and escaping set in the settings of holomorphic semigroup are defined as follows.
\begin{dfn}[\textbf{Fatou set, Julia set and escaping set}]\label{2ab} 
\textit{Fatou set} of the holomorphic semigroup $S$ is defined by
  $$
  F (S) = \{z \in \mathbb{C}: S\;\ \textrm{is normal in a neighborhood of}\;\ z\}
  $$
and the \textit{Julia set} $J(S) $ of $S$ is the compliment of $ F(S) $. If $ S $ is a transcendental semigroup, the \textit{escaping set} of $S$ is defined by 
$$
I(S)  = \{z \in \mathbb{C}: S \;  \text{is iteratively divergent at} \;z \}
$$
We call each point of the set $  I(S) $ by \textit{escaping point}.        
\end{dfn} 
It is obvious that $F(S)$ is the largest open subset (of $\mathbb{C}$ or $ \mathbb{C}_{\infty} $) on which the family $\mathscr{F} $ in $S$ (or semigroup $ S $ itself) is normal. Hence its compliment $J(S)$ is a smallest closed set for any  semigroup $S$. Whereas the escaping set $ I(S) $ is neither an open nor a closed set (if it is non-empty) for any semigroup $S$. Any maximally connected subset $ U $ of the Fatou set $ F(S) $ is called a \textit{Fatou component}.  
        
If $S = \langle f\rangle$, then $F(S), J(S)$ and $I(S)$ are respectively the Fatou set, Julia set and escaping set in classical complex dynamics. In this situation we simply write: $F(f), J(f)$ and $I(f)$. 
 
There is  possibility of being Fatou set, Julia set and escaping set of holomorphic semigroup respectively equal to the Fatou set, Julia set and escaping set of its subsemigroup. To get this results, we need the notion of  different indexes  of subsemigroup of a semigroup $ S $.

\begin{dfn}[\textbf{Finite index and cofinite index}]\label{1fi}
A subsemigroup $ T $ of a holomorphic semigroup $ S $ is said to be of \textit{finite index} if there exists  finite collection of elements $ \{f_{1}, f_{2}, \ldots, f_{n} \}$ of $ S^{1} $ where $S^{1} = S \cup \{\text{Identity}\} $ such that 
\begin{equation}\label{1.1}
 S = f_{1}\circ T \cup f_{2}\circ T\cup \ldots \cup f_{n}\circ T  
\end{equation}
The smallest $ n $ that satisfies \ref{1.1} is called \textit{index} of  $ T $  in $ S $.  
Similarly a subsemigroup $ T $ of a holomorphic semigroup $ S $ is said to be of \textit{cofinite index} if there exists  finite collection of elements $ \{f_{1}, f_{2}, \ldots, f_{n} \}$ of $S^{1} $ such that for any $ f\in S $, there is $ i\in \{1, 2, \ldots, n \} $ such that
\begin{equation}\label{1.2}
f_{i}\circ f \in T 
\end{equation} 
The smallest $ n $ that satisfies \ref{1.2} is called \textit{cofinite index} of $ T $  in $ S $.
\end{dfn}
Note that the size of subsemigroup $ T $ inside semigroup $ S $ is measured in terms of index.  If subsemigroup $ T $ has finite index or cofinite index in semigroup $ S $, then we say $ T $ is finite subsemigroup or cofinite subsemigroup respectively. 

In {\cite[Theorems 5.1]{poo}}, K.K. Poon  proved that Fatou set and Julia set of finitely generated abelian transcendental semigroup $ S $ is same as the Fatou set and Julia set of each of its particular function if semigroup $ S $ is generated by finite type transcendental entire maps. In  {\cite[Theorems 3.3]{sub2}}, we proved that escaping set of transcendental semigroup $ S $ is same as escaping set of each of its particular function if semigroup $ S $ gererated by finite type transcendental entire maps. In this paper, we prove the following result.
\begin{theorem}\label{fin1}
If a subsemigroup $ T $ has finite index or  cofinite index  in an abelian transcendental semigroup $ S $, then $ I(S) = I(T), J(S) =J(T)$ and $ F(S) = F(T) $. 
\end{theorem}
In section 2, we also define another notion of index which is called Rees index. We also proved that if subsemigroup $ T $ has finite Rees index in semigroup $ S $, then $ I(S) = I(T), J(S) =J(T)$ and $ F(S) = F(T) $.

From {\cite[Theorem 3.1 (1) and (3)]{sub1}}, we can say that Fatou set and escaping set of holomorphic semigroup may be empty.  The result {\cite[Theorems 5.1]{poo}} is one of the case of non-empty Fatou set and that of {\cite[Theorems 3.3]{sub2}} is a case of the non-empty escaping set of transcendental semigroup. We see another case of non-empty Fatou set and escaping set on the basis of the following definition.

\begin{dfn}[\textbf{Partial fundamental set and fundamental set}]\label{1ka}
A set $ U $ is a \textit{partial fundamental set} for the semigroup $ S $ if
\begin{enumerate}
\item $ U\neq\emptyset $,
\item $ U\subset R(S) $,
\item $ f(U)\cap U =\emptyset $ for all $ f \in S $.
\\
If in addition to $(1)$, $ (2) $ and $(3)  \;  U $ satisfies the property
\item $\bigcup_{f\in S}f(U) = R(S)$, 
\end{enumerate}
then $ U $ is called \textit{fundamental set} for $ S $. 
\end{dfn}
The set $ R(S) $ is defined and discussed in section 4 of remark \ref{r}. 
From the statements $F(S)\subset \bigcap_{f\in S}F(f)$ and $I(S)\subset \bigcap_{f\in S}I(f)$ ({\cite[Theorem 3.1]{sub1}}), we can say that the Fatou set and the escaping set of holomorphic semigroup may be empty.
On the basis of the definition \ref{1ka}, we prove the following result. 
\begin{theorem}\label{fs1}
 Let $ S $ be holomorphic semigroup and $ U $ is a partial fundamental set for $ S $, then $ U\subset F(S) $. If, in addition, $ S $ be transcendental semigroup and $ U $ is a fundamental set, then $ U\subset I(S) $.
\end{theorem}
The organization of this paper is as follows: In section 2, we briefly review notion of finite index and co-finite index with suitable examples and we review some results from rational (sub) semigroup dynamics and we extend the same in transcendental (sub) semigroup dynamics. We introduce Rees index of subsemigroup  and we prove the  dynamical similarity of holomorphic semigroup and its subsemigroup. In section 3, we prove theorem \ref{fin1} and we also prove  theorem\ref{fin1} by loosing the condition of abelian if the subsemigroup has finite Rees index. In section 4, we define discontinuous transcendental semigroup and on the basis of this notion, we discuss partial fundamental set and fundamental sets and then we prove theorem \ref{fs1}.   

\section{Results from general holomorphic (sub) semigroup dynamics}
There are various notions of investigation of how large a substructure is inside of an algebraic object in the sense  of sharing  properties and structures.  One of the such a notion is \textit{index} and it is an outstanding idea in general group theory and semigroup theory.  It occurs in many important theorems of group theory and semigroup theory.  The notion of finite index, cofinite index and Rees index of subsemigroup was used to compare how much the size of subsemigroup is large enough in semigroup. If the subsemigroup $ T $ is big enough in semigroup $ S $, then  $ S $ and $ T $ share  many properties. In this context,  our proposed theorem \ref{fin1} states that if $ T $ has finite index or cofinte index in $ S $, then both $ S $ and $ T $ share  the same Fatou set, Julia set and escaping set. In semigroup theory, cofinite index is also known as \textit{Grigorochuk index} and this index was introduced by Grigorochuk \cite{gri} in 1988. Note that  $ T $  is cofinite subsemigroup  of a semigroup $ S $ if it has a cofinite index in  $ S $. Maltcev and Ruskuc {\cite[Theorem 3.1]{malt}} proved that if  for every $ f \in S $ of a finitely generated semigroup and every proper cofinite subsemigroup $ T $, then  $ f \circ T \neq S $. 
Note that if semigroup were a group, the notion of finite index and cofinite index coincide. The subsemigroup $ T $ of a finitely  generated semigroup $ S $ consisting of all words  of finite (some multiple of integer n) length (compositions of finite number of holomorphic functions) has finite index and cofinite index in $ S $.  For instance, for any holomorphic function $ f $, the subsemigroup $ \langle f^{n}\rangle $ always has finite index and cofinite index in a semigroup $ \langle f\rangle $. 

We first see an alternative form of finite index and cofinite index of any subsemigroup of holomorphic semigroup. Let $ T $ be a subsemigroup of holomorphic semigroup $ S $. For any $ f \in S^{1} $ where $ S^{1} =S\cup \{Identity\} $, the set of form $f\circ T $ ( or $T\circ f $) is called \textit{translate} of $ T $ by the function $ f $. Let us define following two types of indexes:
\begin{enumerate}
\item The \textit{transnational index} of $ T $ in $ S $ is the number of distinct translates $ f\circ T $ of $ T $ in $ S $.
\item The \textit{strong orbit transnational index} of $ T $ in $ S $ is the number of distinct translates $ f\circ T $ of $ T $ in $ S $ such that $ f\circ g \circ T = T $ for some $ g \in S^{1} $. 
\end{enumerate}
We prove the following result which show that finite index and transnational index are equivalent and cofinite index and strong orbit transnational index are equivalent. 
\begin{theorem}\label{eqt1}
Let $ T $ be a subsemigroup of holomorphic semigroup $ S $. Then 
\begin{enumerate}
\item $ T $ has finite index in $ S $ if and only if it has transnational index.
\item $ T $ has cofinite index in $ S $ if and only if it has strong orbit transnational index.
\end{enumerate}
\end{theorem}
\begin{proof}
The proof is clear from definitions. 
\end{proof}
From this theorem \ref{eqt1} and definition \ref{1fi}, the finite index and cofinite index of subsemigroups of the following examples will be clear. 
\begin{exm}{\cite[Example in page 362]{hin}} \label{ex1}
The subset $ T = \langle f^{2}, \; g^{2}, \;  f\circ g,\;  g\circ f\rangle $ of semigroup $ S = \langle f, \; g\rangle $ is a subsemigroup of $ S $ and it has finite index 3 and cofinite index 2 in $ S $. 
\end{exm}

The more concrete example but similar to the above example \ref{ex1} is as follows. 
The semigroup $ S = \langle \sin z, \cos z\rangle $ has a subsemigroup 
$$
 T = \langle \sin \text{sin z}, \; \cos \text{cos z},\; \sin \text{cos z}, \; \cos \text{sin z} \rangle 
 $$
 Let us denote $ f_{1} = id, \;  f_{2} = \sin z $ and $ f_{3} = \cos z$. Then for any $ f \in S $, we have $ f = f_{i}\circ h \in f_{i} \circ T$ for some $ h \in T $. The number distinct translates of  $ T $ in $ S $  are $f_{1} \circ T =T, f_{2}\circ T $ and $f_{3}\circ T $. So $ S = T \cup f_{2}\circ  T\cup f_{3}\circ T $. This shows that $ T $ has finite index 3 in $ S $. 

Furthermore, if $ f \in S $, then  $ f = h $ or $ f = \sin z \circ h $ or $ f = \cos z \circ h $ for $ h \in T $. Let us choose $ f_{1} = id $ and $ f_{2} = \sin z\;  \text{or}\; \cos z $. If $ f = h $, then $ f_{1} \circ h  = id \circ h = h\in T $. If $ f = \sin z \circ h $, then $ f_{2} \circ f = \sin z \circ \sin z \circ h = \sin \text{sin z} \circ h\in T $. If $ f = \cos z \circ h $, then $ f_{2} \circ f = \cos z \circ \sin z \circ h = \cos \text{sin z} \circ h\in T $. We can choose other combinations, but anyway, we get element of semigroup $ T $. This shows that cofinite index of $ T $ in $ S $ is 2.
\begin{exm}
Let $ S = \langle f, g \rangle $ and $ T = \{\text{words (composition) begining  with} \; f \} $. $ T $  has no finite index in $ S $. The only cofinite subsemigroup of $ T $ is $ T $ itself. So $ T $ has cofinite index 1 in $ S $.  Note that $ S $ finitely generated but $ T $ is not. Since any generating set of $ T $ must contain $ \{f \circ g^{n}: n \geq 1 \} $.  
\end{exm}
\begin{exm}\label{ex3}
Let $ S = \langle f \rangle $ where $ f $ is a holomorphic map and $ T = \langle f^{n} \rangle $. 
The subsemigroup $ T $ has  n-different translates  in $ S $, which are $ T, f\circ T, \ldots ,f^{n-1}\circ T $. So $ T  $ has finite index n in $ S $. In this case, the only cofinite subsemigroup of $ T $ is $ T $ itself. So $ T $ has cofinite index 1 in $ S $.
\end{exm}
In example \ref{ex3}, if we choose subsemigroup of $ S $ as a $ S $ itself, then there are infinitely many translates of $ S $, namely, $ h \circ S = h \circ \langle f\rangle $ for all $ h \in S $.  So, $ S $ has no finite index in itself. Again, it has cofinite index 1 in itself.

From the theorem 3.1 of \cite{sub1}, we can prove the following result:
\begin{lem}\label{fi3}
For any subsemigroup $ T $ of a holomorphic semigroup $ S $, we have $F(S) \subset F(T), J(S) \supset J(T)$. 
\end{lem}
\begin{proof}
We prove $F(S) \subset F(T) $. The next one $ J(T) \supset J(S) $ will be proved by taking the complement of $F(S) \subset F(T)$ in $ \mathbb{C} $.
By the theorem 3.1 of \cite{sub1}, $ F(S) \subset \cap_{f \in S } F(f) $ and $ F(T) \subset \cap_{g \in T} F(g) $ for any subsemigroup $ T $ of semigroup $ S $. Since any $ g \in T $ is also in $ S $, so by the same theorem 3.1 of \cite{sub1}, we also have $ F(S) \subset F(g) $ for all $ g \in T $ and hence $ F(S) \subset \cap_{g \in T } F(g) $. Now for any $ z \in F(S) $, we have $ z \in \cap_{g \in T } F(g) $  for all $ g \in T $. This implies $ z \in F(g) $ for all $ g \in T $.  This proves $ z \in F(T) $ and hence $ F(S) \subset F(T) $. 
\end{proof}

 Hinkannen and Martin {\cite [Theorem 2.4]{hin}} proved that if a subsemigroup $ T $ has finite index or cofinite index in the rational semigroup $ S $, then $ F(S) = F(T) $ and $ J(S) =  J(T) $. In the following theorem, we prove the same result in the case of general holomorphic semigroup. Note that in our study, by general holomorphic semigroup, we mean either it is rational semigroup or a  transcendental semigroup. 
\begin{theorem}\label{th1}
If a subsemigroup $ T $ has finite index or  cofinite index in the holomorphic semigroup $ S $, then $ F(S) = F(T) $ and $ J(S) =  J(T) $.
\end{theorem}
\begin{proof}[Sketch of the proof]
From the above lemma \ref{fi3}, $ F(S) \subset F(T) $ for any holomorphic semigroup $ S $.  If $ S $ is a rational semigroup, the result follows from {\cite [Theorem 2.4]{hin}}. We prove other inclusion if $ S $ is a transcendental semigroup.

Let subsemigroup $ T $ of a semigroup $ S $ has finite index $ n $, then by the definition \ref{1fi},  there exists  finite collection of elements $ \{f_{1}, f_{2}, \ldots, f_{n} \}$ of $ S \cup \{\text{Identity}\} $ such that 
$$
S  = f_{1}\circ T \cup f_{2}\circ T\cup \ldots \cup f_{n}\circ T 
$$
Then for any $ g \in S $, there is $ h \in T $ such that $ g = f_{i} \circ h $.
Choose a sequence $ (g_{j})_{j \in \mathbb{N}} $ in $ S $, then each $ g_{j} $ is of the form  $ g_{j} = f_{i} \circ h_{j} $, where $ h_{j} \in T $,  $ 1 \leq i \leq n $. Here, we may assume same $ i $ for all $ j $.  So without loss of generality, we may choose a subsequence $ (g_{j_{k}}) $ of $ (g_{j}) $ such that $ g_{j_{k}} = f_{i} \circ h_{j_{k}} $ for particular $ f_{i} $,  where $ (h_{j_{k}})$ is a subsequence of $ (h_{j}) $ in  $T$. Since on $ F(T) $, the sequence $ (h_{j_{k}})$ has a convergent subsequence so do the sequences $(g_{j_{k}}) $ and $(g_{j}) $  in $ F(S) $. This proves $ F(T) \subset F(S) $. 

Let subsemigroup $ T $ of a semigroup $ S $ has cofinite index $ n $, then by the definition \ref{1fi},  there exists  finite collection of elements $ \{f_{1}, f_{2}, \ldots, f_{n} \}$ of $ S \cup \{\text{Identity}\} $ such that for every $ f\in S $, there is  $ i\in \{1, 2, \ldots, n \} $ such that
$f_{i}\circ f \in T $. 
Let us choose a sequence $ (g_{j})_{j \in \mathbb{N}} $ in $ S $, then for each $ j $, there is a $ i $ with $ 1 \leq i \leq n $  such that $ f_{i} \circ g_{j} = h_{j} \in T $. Let $ z \in F(T) $. Then the sequence $ h_{j} $  has convergent subsequence in $ T $ so does the sequence  $ (g_{j}) $ in $ F(S) $. This proves $ F(T) \subset F(S) $. 

\end{proof} 

Next, we see a special  subsemigroup  of holomorphic semigroup that yields cofinite index. 
\begin{dfn}[\textbf{Stablizer,  wandering component and  stable domains}]\label{1g}
For a holomorphic  semigroup $ S $, let $ U $ be a component of the  Fatou set $ F(S) $ and $ U_{f} $ be a component of Fatou set containing $ f(U) $ for some $ f\in S $.  The set of the form 
$$S_{U} = \{f\in S : U_{f} = U\}  $$
is called \textit{stabilizer} of $ U $ on $ S $. If $ S_{U} $ is non-empty,  we say that a component $ U $ satisfying  $U_{f} = U  $ is called \textit{stable basin} for  $ S $. The component $ U $ of $ F(S) $ is called wandering if the set $ \{U_{f}: f \in S \} $ contains infinitely many elements. That is, $ U $ is a wandering domain if there is sequence of elements $ \{f_{i}\} $ of $ S $ such that $ U_{f_{i}}  \neq U_{f_{j}}$ for $ i \neq j $. Furthermore, the component $ U $ of $ F(S) $ is called strictly wandering if $U_{f} = U_{g} $ implies $ f =g $. A stable basin $ U $ of a holomorphic semigroup $ S $ is
\begin{enumerate}
\item  \textit{attracting} if it is a subdomain of attracting basin of each $ f\in S_{U} $
\item  \textit{supper attracting} if it is a subdomain of supper attracting basin of each $ f\in S_{U} $
\item  \textit{parabolic} if it is a subdomain of parabolic basin of each $ f\in S_{U} $
\item  \textit{Siegel} if it is a subdomain of Siegel disk of each $ f\in S_{U} $
\item  \textit{Baker} if it is a subdomain of Baker domain  of each $ f\in S_{U} $
\item  \textit{Hermann} if it is a subdomain of Hermann  ring of each $ f\in S_{U} $
\end{enumerate}
\end{dfn}
In classical holomorphic iteration theory, the stable basin is one of the above types but in transcendental iteration theory, the stable basin is not a Hermann because a transcendental entire function does not have Hermann ring {\cite[Proposition 4.2] {hou}}.

Note that for any rational function  $ f $, we always have $ U_{f} = U $. So $ U_{S} $ is non-empty for a rational semigroup $ S $.  However,  if $ f $ is transcendental, it is possible that $ U_{f} \neq U $. So, $ U_{S} $ may be empty for transcendental semigroup $ S $.  Bergweiler and Rohde \cite{ber} proved that $ U_{f} - U $ contains at most one point which is an asymptotic value of $ f $ if  $ f $  is an entire function. Note that value in $ U_{f} - U $ need not be omitted value.  For example, the transcendental entire function $ f(z) = ze^{-(1/2 z^{2}+ 3/2 z -1)} $ has an attracting fixed point $ 0 $. Since $ f(z) \to 0  $ as $ n\to \infty $, so $ 0 $ is an asymptotic value of $ f $. If we let $ U $ a component of Fatou set $ F(f) $ that contains all large positive real numbers, then $ 0\notin f(U) $. There is a  Fatou component $ U_{f} $ containing $ f(U) $ that contains 0.

\begin{lem}\label{ss1}
Let $ S $ be a holomorphic  semigroup. Then the stabilizer $S_{U}$ (if it is non-empty) is a subsemigroup of $ S $ and $ F(S) \subset F(S_{U}), \; J(S) \supset J(S_{U}) $. 
\end{lem}
\begin{proof}
Let $ f, g \in S_{U} $, then by the definition \ref{1g}, $U_{f} = U  $ and $U_{g} = U  $. Since $ U_{f} $ and  $ U_{g} $ are  components of Fatou set containing $ f(U) $ and $ g(U) $ respectively. That is, $ f(U) \subseteq U_{f} = U  $ and $ g(U) \subseteq U_{g} = U \Longrightarrow (f\circ g)(U) = f(g(U))\subseteq f(U_{g}) = f(U)\subseteq U_{f} = U$. Since $(f \circ g)(U) \subseteq U_{f \circ g} $. Then either $ U_{f\circ g} \subseteq U$ or $ U \subseteq U_{f\circ g} $. The only possibility in this case is $ U_{f\circ g} = U$. Hence $ f\circ g \in S_{U} $, which proves that $S_{U}$ is a subsemigroup of $ S $. The proof of $ F(S) \subset F(S_{U}), \; J(S) \supset J(S_{U})$ follows from lemma \ref{fi3}.
\end{proof}

There may be a connection between no wandering domains and the stable basins of cofinite index. We have proposed the connection in the following statement for general holomorphic semigroup $ S $.

\begin{theorem}\label{ss2}
 Let $ S $ be a holomorphic  semigroup with no wandering domains. Let $ U $ be any component of Fatou set. Then the forward orbit $ \{U_{f}: f\in S\} $ of $ U $ under $ S $ contains a stabilizer of $ U $ of cofinite index.
\end{theorem}
\begin{proof}
If $ S $ is a rational semigroup, the proof see for instance in {\cite[Theorem 6.1]{hin}}.
 
If $ S $ is a transcendental semigroup, we sketch our proof in the following way.

We have given that $ U $ be a non-wandering component of Fatou set $ F(S) $. So  $ U $ has a finite forward orbit  $ U_{1}, U_{2}, \ldots , U_{n} $ (say) with $ U_{1} = U $. \\
Case (i): If for every $ i =1,2,\ldots n $,  there is $ f_{i} \in S $  such that $ f_{i}(U_{i})\subseteq U_{1} $, then by the above lemma \ref{ss1},   stabilizer  $S_{U_{1}} = \{f\in S : U_{1}{_{f}} = U_{1}\}$ is a subsemigroup of $ S $. For any $ f \in S $ there is $ f_{i} $ for each $i =1,2, \dots, n $ such that $ U_{1_{f_{i}\circ f}} = U_{1}$.  Which shows that $ f_{i} \circ f \in S_{U_{1}} $. Therefore, $ U_{1} $ is a required stable basin such that the stabilizer  $S_{U_{1}}$ has cofinite index in $ S $. \\
Case (ii): If for every $ j = 2,\ldots n $,  there is $ f_{j} \in S $  such that $ f_{j}(U_{j})\subseteq V $, where $ V = U_{j} $ such that $ j\geq 2 $, then number of components of forward orbits of $ V $ is strictly less than $ U $. By this way, we can find a component $ W = U_{i} $ for some $ 1\leq i \leq n $ whose forward orbit has fewest components. 
For every component $ W_{g} $ of the forward orbit of $ W $, there is a $ f\in S $ such that $ f(W_{g})\subseteq W $. That is,  $ W_{g \circ f} = W $, and it follows that $ W $ is a required stable basin  such that the stabilizer $S_{W}$ has cofinite index.
\end{proof}
Note that in our forth coming study, the subsemigroup $ S_{U} $ of cofinite index in $ S $ is replaced by  a basin $ U $ of cofinite index. In this sense, above
theorem \ref{ss2} is a criterion to have a basin $ U $ of cofinite index. 
Similar to rational semigroups {\cite[Conjectures 6.1 and 6.2]{hin}}, we have hoped that the following analogous two statements will also be conjectures in the case of transcendental semigroup $ S $. Note that each is true in classical complex dynamics (rational and transcendental).  
\begin{conj}
Let $ S $ be a (finitely generated) transcendental semigroup such that $ F(S) \neq \emptyset $. Then  a stable basin $ U $  has cofinite index in $ S $.
\end{conj}

\begin{conj}
Let $ S $ be a finitely generated transcendental semigroup. Then for each component $ U $ of $ F(S) $, there is a stable basin $ V $ for $ S $ lying in the forward orbit of $ U $  has  cofinite index in $ S $.
\end{conj}
 Note that there are examples of holomorphic semigroups whose subsemigroups  have cofinite index but not have finite index. For example, for any $ f \in S $ (holomorphic semigroup), the sets $ S \circ f = \{g \circ f: g \in S\} $  and $ f \circ S = \{f \circ g: g \in S\} $  are subsemigroups (in fact, left and right ideals (see for instance in {\cite[Definition 2.2 and Proposition 2.1]{sub2}})) of $ S $. Each of these subsemigroups has cofinite index 1 in $ S $ but not have finite index in $ S $.There may be a lot of  subsemigroups of holomorphic semigroup $ S $ that may have finite index in $ S $. If we able to find such subsemigroups, there will be a chance of replacing cofinite index by finite index and so further investigations will be more interesting. We leave this notion of investigation for future research. 
 
From the example just we mentioned in above paragraph, we can also say that $ S \circ f $  and  $ f \circ S$ are not finitely generated even if the semigroup $ S $ is. For if $ S \circ f = \langle f_{1}, f_{2}, \ldots, f_{n} \rangle$
where $ f_{i} \in S$ for $ i=1,2,\ldots n $, then $ f_{i}=  g_{i} \circ f$, where $ g_{i} \in S $. For any $ g \in S $, we have $ g^{n}\circ f \in S\circ f $ for all $ n\geq 1 $ but not every $ g^{n}\circ f \in \langle f_{1}, f_{2}, \ldots, f_{n} \rangle$.  This fact together with discussion in above paragraph, we came to know that the notion of cofinite index fails to preserve basic finiteness (finitely generated) condition of subsemigroup. That is, if $ T $ is a subsemigroup of cofinite index in semigroup $ S $, then $ S $ is finitely generated may not always imply $ T $ is finitely generated. There is another notion of index which preserves finiteness condition of subsemigroup. 
\begin{dfn}[\textbf{Rees index}]\label{ri}
Let $ S $ be a semigroup and $ T $ be its subsemigroup. The Rees index of $ T $ in $ S $ is defined as $ |S-T| +1 $. In this case, we say $ T $ is large subsemigroup of $ S $ and $ S $ is small extensions of $ T $. 
\end{dfn}

This definition was first introduced by A. Jura \cite{jur} in the case when $ T $ is an ideal of semigroup $ S $. In such a case, the Rees index of $ T $ in $ S $ is the cardinality of factor semigroup $ S/T $. From the definition \ref{ri}, it is clear that the Rees index of $ T $ in $ S $ is the size of the compliment $ S-T $. To have Rees index of any subsemigroup in its parent semigroup is fairly a restrictive property, and it occurs naturally in semigroups ( for instance all ideals in additive semigroup of positive integers are of finite Rees index). Note that Rees index does not generalize group index, even the notion of finite Rees index does not generalize finite group index. That is, if $ G $ is an infinite group and $ H $ is proper subgroup, then group index of $ H $ in $ G $ is  finite but Rees index is infinite. In fact, let $ G $ be an infinite group and $ H $ is a subgroup of $ G $, then $ H $ has finite Rees index in $ G $ if and only if $ H = G $. 

Next, we investigate how much similar a  semigroup $ S $ and its large subsemigroup $ T $ are. One of the basic similarity (proved first by Jura \cite{jur}) is the following result. 
\begin{theorem}\label{ri1}
Let $ T $ be a large subsemigroup of semigroup $ S $. Then $ S $ is finitely generated if and only if $ T $ is. 
\end{theorem}
\begin{proof}
See for instance {\cite[Theorem 1.1]{rus}}. 
\end{proof}
On the basis of this theorem \ref{ri1}, we proof the following dynamical similarity of holomorphic semigroup and its subsemigroup.
\begin{theorem}\label{ri2}
Let $ T $ be a large subsemigroup of finitely generated holomorphic semigroup $ S $. Then $ F(S)  =F(T) $ and $ J(S) = J(T).  $
\end{theorem}
\begin{proof}
We prove $ F(S) = F(T) $, another one is clear by taking compliment.
By the lemma \ref{fi3}, it is clear that $ F(S) \subset F(T) $. So, we prove only $ F(T) \subset F(S) $. By above theorem \ref{ri1}, $ T $ is finitely generated and let $ X = \{f_{1},  f_{2}, \ldots , f_{n} \} \subset S$  be a generating set of $ T $. Then clearly $ S $ is generated by the set $ Y = X \cup (S-T) $.  Every sequence $ (f_{i}) $ in $ F(T) $ where  $ f_{i} = f_{i_{1}} \circ f_{i_{2}} \circ \ldots \circ f_{i_{n}}$ has a convergent subsequence.  Now each element $ g_{m} $ of a sequence $ (g_{m}) $ in $ S $ can be written as  $ g_{m} = f_{i_{1}} \circ f_{i_{2}} \circ \ldots \circ f_{i_{n}}  \circ h_{j_{1}} \circ h_{j_{2}} \circ \ldots \circ h_{j_{k}}$, where $ S - T = \{h_{1},  h_{2}, \ldots , h_{k} \} \subset S$. Since $ S-T $ is finite, so convergent sequence in $ F(T) $ can be finitely extended to the convergent sequence in $ F(S) $. So, every sequence  $ (g_{m}) $ in $ F(S) $ has a convergent subsequence. Hence $ F(T) \subset F(S) $. 
\end{proof}

\section{Proof of the Theorem \ref{fin1}}
In this section, we concentrate on our mission of proving theorem \ref{fin1}.
We first prove the  analogous result of lemma \ref{fi3} in the case of escaping set of transcendental semigroup.  
\begin{lem}\label{fi4}
For any subsemigroup $ T $ of a transcendental  semigroup $ S $, we have   $I(S) \subset I(T) $.
\end{lem}
\begin{proof}

By the theorem 3.1 of \cite{sub1}, $ I(S) \subset \cap_{f \in S } I(f) $ and $ I(T) \subset \cap_{g \in T} I(g) $ for any subsemigroup $ T $ of semigroup $ S $. Since any $ g \in T $ is also in $ S $, so by the same theorem 3.1 of \cite{sub1}, we also have $ I(S) \subset I(g) $ for all $ g \in T $ and hence $ I(S) \subset \cap_{g \in T } I(g) $. Now for any $ z \in I(S) $, we have $ z \in \cap_{g \in T } I(g) $  for all $ g \in T $. This implies $ z \in I(g) $ for all $ g \in T $. By the definition \ref{2ab}, we have $ g^n(z) \to \infty $ as $ n \to \infty $ for all  $ g \in T $. This proves $ z \in I(T) $ and hence $ I(S) \subset I(T) $.

\end{proof}

\begin{lem}\label{3}
Let $S$ be a transcendental  semigroup. Then
\begin{enumerate}
\item $int(I(S))\subset F(S)\;\ \text{and}\;\ ext(I(S))\subset F(S) $, where $int$ and $ext$ respectively denote the interior and exterior of $I(S)$.  
 \item $\partial I(S) = J(S)$, where $\partial I(S)$ denotes the boundary of $I(S)$. 
\end{enumerate}
\end{lem}
\begin{proof}
 We refer for instance lemma 4.2 
and  theorem 4.3  of \cite{kum2}.
\end{proof}
Note that this lemma \ref{3} is a extension of Eremenko's  result \cite{ere},  $\partial I(f) = J(f)$ of classical transcendental dynamics to more general  semigroup settings. 
\begin{proof}[Proof of the Theorem \ref{fin1}]
We prove $ I(S) = I(T) $.  The proof of $J(S) = J(T) $ is obvious from above lemma \ref{3} (2). The fact $ F(S) =F(T) $ is also obvious. 
By the above lemma \ref{fi3}, we always have $ I(S) \subset I(T)   $ for any subsemigroup $ T $ of semigroup $ S $.  For proving this theorem it is enough to show the opposite inclusion $ I(T) \subset I(S) $. 

Let subsemigroup $ T $ of a semigroup $ S $ has finite index $ n $, then by the definition \ref{1fi},  there exists  finite collection of elements $ \{f_{1}, f_{2}, \ldots, f_{n} \}$ of $ S \cup \{\text{Identity}\} $ such that 
$$
S  = f_{1}\circ T \cup f_{2}\circ T\cup \ldots \cup f_{n}\circ T 
$$
Then for any $ g \in S $, there is $ h \in T $ such that $ g = f_{i} \circ h $.
Choose a sequence $ (g_{j})_{j \in \mathbb{N}} $ in $ S $, then each $ g_{j} $ is of the form  $ g_{j} = f_{i} \circ h_{j} $, where $ h_{j} \in T $,  $ 1 \leq i \leq n $. Here we may assume same $ i $ for all $ j $.  
Let $ z \in I(T) $, then by {\cite[Theorem 2.2]{sub1}}, every sequence $( h_{j})_{j \in \mathbb{N}} $ in $ T $  has a divergent subsequence $( h_{j_{k}})_{j_k \in \mathbb{N}} $. That is,  $ h_{j_{k}}^{n}(z) \to \infty $ as $ n\to \infty $ for all $ j_{k} $. 
In this case,  every sequence $(g_{j})_{j \in \mathbb{N}}$ in $ S $ have subsequence $ (g_{j_{k}})_{k \in \mathbb{N}} $, where $ g_{j_{k}} = f_{i} \circ h_{j_{k}} $  with  $ h_{j_{k}}^{n}(z) \to \infty $ as $ n \to \infty $. 
Since $ S $ is an abelian transcendental semigroup, so  $ g_{j_{k}} = f_{i} \circ h_{j_{k}}  = h_{j_{k}} \circ f_{i}$. Thus we may write $g_{j_{k}}^{n}(z)  = h_{j_{k}}^{n} ( f_{i}(z))  \to \infty $   as $ n\to \infty $. This shows that $f_{i}(z) \in I(S) $. If $ f_{i} $ = identity for particular $ i $, we are done. If $ f_{i} $ is not identity, then of course, it is an element of abelian transcendental semigroup $ S $ and  in this case, $ I(S) $ is backward invariant by {\cite[Theorem 2.6]{sub2}}.  So we must have $z \in I(S) $. Therefore, $ I(T) \subset I(S) $.   
 
Let subsemigroup $ T $ of a semigroup $ S $ has cofinite index $ n $, then by the definition \ref{1fi},  there exists  finite collection of elements $ \{f_{1}, f_{2}, \ldots, f_{n} \}$ of $ S \cup \{\text{Identity}\} $ such that for every $ f\in S $, there is  $ i\in \{1, 2, \ldots, n \} $ such that
$f_{i}\circ f \in T $. 
Let us choose a sequence $ (g_{j})_{j \in \mathbb{N}} $ in $ S $, then for each $ j $, there is a $ i $ with $ 1 \leq i \leq n $  such that $ f_{i} \circ g_{j} = h_{j} \in T $. Let $ z \in I(T) $, then by the same argument stated in the first part, every  sequence $( h_{j})_{j \in \mathbb{N}} $ in $ T $  has a divergent subsequence $( h_{j_{k}})_{j_k \in \mathbb{N}} $ at point $ z $. This follows that sequence  $( f_{i} \circ g_{j})$ has a divergent subsequence $( f_{i} \circ g_{j_{k}})$ (say)  at $ z $. Since $ S $ is abelian, so we can write $ ( f_{i} \circ g_{j_{k}})(z) = ( g_{j_{k}} \circ f_{i})(z) =   g_{j_{k}}(f_{i}(z)) = h_{j_{k}}(z)$. Now for any $ z \in I(T) $,  $ h_{j_{k}} \in T $, we must have $h_{j_{k}}^{n}(z) = g_{j_{k}}^{n}(f_{i}(z)) \to \infty$ as $ n \to \infty $. This implies that  $ f_{i}(z) \in I(S) $. If $ f_{i} $ = identity for particular $ i $, we are done. If $ f_{i} $ is not an identity, then of course, it is an element of abelian transcendental semigroup $ S $, then by the same fact explained in the last part of above paragraph, we have $ I(T) \subset I(S) $. 
\end{proof}
The condition of abelian can be loosed from the theorem \ref{fin1} if we choose Rees index. So we can generalize this  theorem to the following result. 

\begin{theorem}
If subsemigroup $ T $ of a finitely generated transcendental semigroup $ S $ has finite Rees index, then $ I(S) = I(T), J(S) =J(T)$ and $ F(S) = F(T) $.
\end{theorem}
\begin{proof}
The last two, that is, $ J(S) =J(T)$ and $ F(S) = F(T) $ was proved in theorem \ref{ri2} of section 2. Also, if we prove $ I(S) = I(T) $, then  $J(S) = J(T) $ is obvious from above lemma \ref{3} (2).
The fact $ I(S) \subset I(T)   $ for any subsemigroup $ T $ of semigroup $ S $ is obvious.  So we prove $ I(T) \subset I(S) $.  

By the theorem \ref{ri1}, $ T $ is finitely generated and let $ X = \{f_{1},  f_{2}, \ldots , f_{n} \} \subset S$  be a generating set of $ T $. Then clearly $ S $ is generated by the set $ Y = X \cup (S-T) $.  By  {\cite[Theorem 2.2]{sub1}},  every sequence $ (f_{i}) $ in $ T $ where  $ f_{i} = f_{i_{1}} \circ f_{i_{2}} \circ \ldots \circ f_{i_{n}}$ for each $ 1\leq i \leq n $ has a divergence subsequence  $ (f_{n_{k}}) $ at each point of $I(T) $.   Now each element $ g_{m} $ of a sequence $ (g_{m}) $ in $ S $ can be written as  $ g_{m} = f_{i_{1}} \circ f_{i_{2}} \circ \ldots \circ f_{i_{n}}  \circ h_{j_{1}} \circ h_{j_{2}} \circ \ldots \circ h_{j_{k}}$, where $S-T =  \{h_{1},  h_{2}, \ldots , h_{k} \} \subset S$ is a finite set.  This show that divergent sequence in $ I(T) $ can be extended finitely to divergent sequence in $ I(S) $. So, every sequence  $ (g_{m}) $ in $ I(S) $ has a divergent subsequence. Hence $ I(T) \subset I(S) $. 
\end{proof}

\section{Proof of the Theorem \ref{fs1}}

It is known to us that for certain holomorphic semigroups,   the Fatou set and the escaping set might be empty. In this section, we discuss notion of discontinuous semigroup. Such type of notion yields partial fundamental set and fundamental set. We prove theorem \ref{fs1} by showing partial fundamental set is in Fatou set $ F(S) $ and fundamental set is in escaping set $ I(S) $. 
\begin{dfn}[\textbf{Discontinuous semigroup}] \label{ds}
A semigroup $ S $ is said to be \textit{discontinuous} at a point $ z\in \mathbb{C} $ if there is a neighborhood  $ U $ of $ z $ such that $ f(U)\cap U =\emptyset $ for all $ f\in S $ or equivalently, translates of $ U $ by distinct elements of $ S $ ($S$-translates) are disjoint. The neighborhood $ U $ is also called nice neighborhood of $ z $. 
\end{dfn}

\begin{rem}\label{r}
Given a holomorphic semigroup $ S $, there are two natural subsets associated with $ S $.
\begin{enumerate}
\item The regular set $ R(S) $ that consists of points  $z\in \mathbb{C} $ at which $ S $ is discontinuous.
\item The limit set $ L(S) $ that consists of points  $z\in \mathbb{C} $ for which there is a point $ z_{0} $, and a sequence $\{f_{n}\}  $ of distinct elements of $ S $ such that $ f_{n}(z_{0})\rightarrow z $ as $ n \to \infty $.
\end{enumerate}
\end{rem}
A set $ X \subset\mathbb{C} $ is S-invariant or invariant under $ S $ if $ f(X) = X $ for all $ f \in S $. It is clear that both of the sets $ R(S) $ and $ L(S) $ are S-invariant. If $ U $ is a nice neighborhood, then every point of $ U $ lies in $ R(S) $. 
So  $ R(S) $ is a open set where as the set  $ L(S) $ the  a close set and $ R(S) \cap L(S) = \emptyset $. 
Recall that a set $ U $ is a \textit{partial fundamental set} for the semigroup $ S $ if
(1) $ U\neq\emptyset $,
(2)  $ U\subset R(S) $,
(3)  $ f(U)\cap U =\emptyset $ for all $ f \in S $.
If in addition to $(1)$, $(2)$ and $(3),  \;  U $ satisfies the property
(4) $\bigcup_{f\in S}f(U) = R(S)$, 
then $ U $ is called \textit{fundamental set} for $ S $. 
We say that $ x, y \in \mathbb{C} $ are S- equivalent if there is a $ f \in S $ such that $ f(x) =y $. 
Above condition $ (3) $ asserts that any two points of $ U $ are not S-equivalent under semigroup $ S $, and condition $ (4) $ asserts that every point of $ R(S) $ is equivalent to some point of $ U $. 
Note that if we replace  $ (3) $ by $f^{-1}(U)\cap U =\emptyset $ for all $ f\in S $, we say $ U $ is a backward partial fundamental set for $ S $ and if, in addition with condition $ (4) $, we say $ U $ is backward partial fundamental set. Note that  the following two theorems \ref{fs1} and \ref{fst2} hold if we have given partial backward fundamental set in the statements.   Similar to the results of Hinkkanen and Martin {\cite[Lemma 2.2]{hin}} in the case of rational semigroup,  we have prove the following in the case of transcendental semigroup $ S $.

\begin{proof}[Proof of the Theorem \ref{fs1}]
Let $ S $ is a holomorphic semigroup. $ U $ is non-empty open set and $ f(U)\cap U =\emptyset $ for all  $ f\in S $ by definition \ref{ds}. The statement $ f(U)\cap U =\emptyset $ for all $ f \in S$ implies that $ S $ omits $ U $ on $ U $. Since $ U $ is open, it contains more than two points. Then by Montel's theorem, $ S $ is normal on $ U $. So, $ U\subset F(S) $. 

Let $ S $ is a transcendental semigroup.To prove $ U\subset I(S) $, we have to show that $ f^{n}(z)\rightarrow \infty $ as $ n \rightarrow \infty $ for all $ f \in S $ and for all $ z \in U $.  The condition $ f(U)\cap U =\emptyset $ for all $ f \in S $ implies that $ f^{n}(U)\cap U =\emptyset $ because as $ f \in S $, then so $ f^{n} \in S $. Also, $ U $ is a fundamental set, so by the definition \ref{1ka} (4), we have  $\bigcup_{f\in S}f(U) = R(S)$. So by the remark \ref{r}(2), there are no points in $ U $ which appear as the limit points under distinct $ (f_{m})_{m \in \mathbb{N}} $ in $ S $. That is, $ (f_{m}) $ has divergent  subsequence $ (f_{m_{k}}) $ at each point of $ U $. Thus, by  {\cite[Theorem 2.2]{sub1}}, for any $ z \in U,  \; f^{n}(z)\rightarrow \infty$ as $ n \rightarrow \infty $ for any $ f \in (f_{m}) $. This  shows that $ U\subseteq I(S) $. 
\end{proof}
Finally, we try to generalize above result (Theorem \ref{fs1}) in the following form. We have given here a short sketch of the proof. For more detail proof, we refer {\cite[Theorem 2.1] {poo1}}.
\begin{theorem}\label{fst2}
Let $ U_{1} $ and $ U_{2} $ be two (partial) fundamental sets for transcendental semigroups $ S_{1} $ and $ S_{2} $ respectively. Suppose furthermore that $ \mathbb{C}\setminus U_{1} \subset U_{2} $ and $ \mathbb{C}\setminus U_{2} \subset U_{1} $. Then the semigroup $ S = \langle S_{1}, S_{2} \rangle $ is discontinuous, and $ U = U_{1} \cap U_{2} $ is a (partial) fundamental set for semigroup $ S $. 
\end{theorem}
\begin{proof}[Sketch of the proof]
Let $ U_{1} $,  $ U_{2} $ and $ S_{1} $,   $ S_{2} $ as given in the statement of the theorem, it is clear from the theorem \ref{fs1} that  $ F(S_{1})\neq \emptyset, \; F(S_{2})\neq \emptyset $ and ($ I(S_{1})\neq \emptyset, \; I(S_{2})\neq \emptyset $ if $ U_{1} $ and $ U_{2} $ are fundamental sets of $ S_{1} $ and $ S_{2} $ respectively). Note that $ U \neq \emptyset  $ by the assumption of the theorem and $ f(U) \cap U  = \emptyset $ for every $ f \in S $ follows easily. This proves $ S $ is discontinuous together with  $ U $ is a (partial) fundamental set for $ S $. 

\end{proof}

\end{document}